%% file: kovacs.tex
\def\year{2021}\relax
\documentclass[letterpaper]{article} 
\usepackage{aaai21}  
\usepackage{times}  
\usepackage{helvet} 
\usepackage{courier}  
\usepackage[hyphens]{url}  
\usepackage{graphicx} 
\usepackage{natbib}  
\usepackage{caption} 
\usepackage{booktabs}       
\usepackage{amsfonts}       
\usepackage{nicefrac}       
\usepackage{microtype}      
\usepackage{amsmath}
\usepackage{multirow}
\usepackage{xcolor}
\usepackage[ruled,vlined]{algorithm2e}
\usepackage{amsthm}
\usepackage{adjustbox}
\usepackage{thmtools}
\usepackage{subcaption}

\newcommand{\mvec}[1]{\boldsymbol{#1}}

\newcommand{\RR}{\mathcal{R}}

\newcommand{\B}{\{0,1\}}

\newcommand{\R}{\mathbb{R}}

\newcommand{\Z}{\mathbb{Z}}

\newtheorem{theorem}{Theorem}
\newtheorem{lemma}[theorem]{Lemma}
\newtheorem{definition}[theorem]{Definition}
\newtheorem{example}[theorem]{Example}

\title{Binary Matrix Factorisation via Column Generation}
\author {
    R\'eka \'A. Kov\'acs,\textsuperscript{\rm 1}
    Oktay G\"unl\"uk, \textsuperscript{\rm 2}
    Raphael A. Hauser \textsuperscript{\rm 1} \\
}
\affiliations {
    \textsuperscript{\rm 1} University of Oxford \& The Alan Turing Institute \\
    \textsuperscript{\rm 2} Cornell University \\
    reka.kovacs@maths.ox.ac.uk, ong5@cornell.edu, hauser@maths.ox.ac.uk
}

\begin{document}

\maketitle

\begin{abstract}
Identifying discrete patterns in binary data is an important dimensionality reduction tool in machine learning and data mining. 
In this paper, we consider the problem of low-rank binary matrix factorisation (BMF) under Boolean arithmetic.
Due to the hardness of this problem, most previous attempts rely on heuristic techniques.
We formulate the problem as a mixed integer linear program and use a large scale optimisation technique of column generation to solve it without the need of heuristic pattern mining. 
Our approach focuses on accuracy and on the provision of optimality guarantees. 
Experimental results on real world datasets demonstrate that our proposed method is effective at producing highly accurate factorisations and improves on the previously available best known results for $15$ out of $24$ problem instances.
\end{abstract}

\section{Introduction}




Low-rank matrix approximation is an essential tool for dimensionality reduction in machine learning. 
For a given $n\times m$ data matrix $X$ whose rows correspond to $n$ observations or items, columns to $m$ features and a fixed positive integer $k$, computing an optimal rank-$k$ approximation consists of approximately factorising $X$ into two matrices $A$, $B$ of dimension $n\times k $ and $k\times m$ respectively, so that the discrepancy between $X$ and its rank-$k$ approximate $A\cdot B$ is minimum. 
The rank-$k$ matrix $A\cdot B$ describes $X$ using only $k$ derived features: the rows of $B$ specify how the original features relate to the $k$  derived features, while the rows of $A$ provide weights how each observation can be (approximately) expressed as a linear combination of the $k$ derived features.

Many practical datasets contain observations on categorical features and while classical methods such as singular value decomposition (SVD) \cite{SVD} and non-negative matrix factorisation (NMF) \cite{NMF} can be used to obtain low-rank approximates for real valued datasets, for a binary input matrix $X$ they cannot guarantee factor matrices $A,B$ and their product to be binary.
Binary matrix factorisation (BMF) is an approach to compute low-rank matrix approximations of binary matrices ensuring that the factor matrices are binary as well \cite{Miettinen:2012:PKDD:Tutorial:On:BMF}.
More precisely, for a given binary matrix $X\in\B^{n\times m}$ and a fixed positive integer $k$, the rank-$k$ BMF problem ($k$-BMF) asks to find two matrices $A\in \B^{n\times k}$ and $B\in \B^{k\times m}$ such that the product of $A$ and $B$ is a binary matrix denoted by $Z$, and the distance between $X$ and $Z$ is minimum in the squared Frobenius norm.
Many variants of $k$-BMF exist, depending on what arithmetic is used when the product of matrices $A$ and $B$ is computed.  
We focus on a variant where the Boolean arithmetic is used:
\\[1mm]
$~~~~~~~~~~~~~~~~X=A\circ B  \iff x_{ij} = \bigvee_{\ell=1}^k a_{i\ell} \wedge b_{\ell j},$\\[2mm]
so that $1$s and $0$s are interpreted as True and False, addition corresponds to logical disjunction ($\vee$) and multiplication to conjunction ($\wedge$). Apart from the arithmetic of the Boolean semi-ring, other choices include standard arithmetic over the integers or modulo $2$ arithmetic over the binary field. 
We focus on the Boolean case, in which the property of Boolean non-linearity, 1 + 1 = 1 holds because many natural processes follow this rule. For instance, when diagnosing patients with a certain condition, it is only the presence or absence of a characteristic symptom which is important, and the frequency of the symptom does not change the 
diagnosis.
As an example, consider the matrix (inspired by \cite{Miettinen:2008:DBP:1442800.1442809})  $$X= \left[\begin{smallmatrix} 1 & 1 & 0\\
1 & 1 & 1\\
0 & 1 & 1\end{smallmatrix}\right]$$
where rows correspond to patients and columns to symptoms, $x_{ij}=1$ indicating patient $i$ presents symptom $j$. Let 
$$X=A\circ B = \left[ \begin{smallmatrix}
1 & 0 \\
1 & 1 \\
0 & 1
\end{smallmatrix} \right] 
\circ 
\left[ \begin{smallmatrix}
1 & 1 & 0 \\
0 & 1 & 1
\end{smallmatrix} \right]$$ denote the rank-2 BMF of $X$. Factor $B$ reveals that 2 underlying diseases cause the observed symptoms, $\alpha$ causing symptoms 1 and 2, and $\beta$ 
causing 2 and 3. Factor $A$ reveals that patient 1 has disease $\alpha$, patient 3 has $\beta$ and patient 2 has both.
In contrast,  the best rank-2 real approximation 
$$X\approx
\left[
\begin{smallmatrix}
1.21&0.71\\
1.21 &0.00\\
1.21 &-0.71\\
\end{smallmatrix}
\right]
\left[
\begin{smallmatrix}
0.00&0.71&0.50\\
0.71& 0.00&-0.71
\end{smallmatrix}
\right]
$$ fails to reveal a clear interpretation, and  the best rank-2 NMF  
$$X\approx\left[
\begin{smallmatrix}
1.36 &0.09\\
1.05 &1.02\\
0.13 &1.34
\end{smallmatrix}\right]
\left[\begin{smallmatrix}
0.80 &0.58 &0.01\\
0.00     &     0.57 &0.81\end{smallmatrix}\right]$$
of $X$ suggests that symptom 2 presents with lower intensity in both $\alpha$ 
and $\beta$, an erroneous conclusion (caused by patient 2) that could not have been learned from data $X$ which is of ``on/off'' type. 
BMF-derived features are particularly natural to interpret in biclustering gene expression datasets \cite{Zhang:2007}, role based access control \cite{Lu:2008:OBM:1546682.1547186,Lu:2014:9400730720140101} and market basket data clustering \cite{Li:2005}.

\subsection{Complexity and Related Work}
\label{comp}
%
The Boolean rank of a binary matrix $X\in\B^{n\times m}$ is defined to be the smallest integer $r$ for which there exist matrices $A\in \B^{n\times r}$ and $B\in \B^{r\times m}$  such that $X=A\circ B$, where $\circ$ denotes Boolean matrix multiplication defined as $x_{ij} = \bigvee_{\ell=1}^r a_{i\ell} \wedge b_{\ell j}$ for all $i\in\{1,\dots,n \}:=[n]$, $j\in[m]$ \cite{Kim:1982}. This is equivalent to $x_{ij} = \min\{1,\sum_{\ell=1}^r a_{i\ell} b_{\ell j}\}$ using standard arithmetic. 
%
Equivalently, the Boolean rank of $X$ is the minimum value of $r$ for which it is possible to factor $X$ into the Boolean combination of $r$ rank-$1$ binary matrices $X = \bigvee_{\ell=1}^r \mvec{a}_\ell \mvec{b}^\top_\ell$ for 
$\mvec{a}_\ell \in\B^{n},\mvec{b}_\ell \in \B^m$. 
Interpreting $X$ as the node-node incidence matrix of a bipartite graph $G$ with $n$ vertices on the left and $m$ vertices on the right, the problem of computing the Boolean rank of $X$ is in one-to-one correspondence with finding a minimum edge covering of $G$ by complete bipartite subgraphs (bicliques)\cite{Monson:1995}. Since the biclique cover problem is NP-hard \cite{Orlin:1977}
and hard to approximate \cite{Simon:1990OnAS}, computing the Boolean rank is hard as well.
Finding an optimal $k$-BMF of $X$ has a graphic interpretation of minimizing the number of errors in an approximate covering of $G$ by $k$ bicliques. 
Even the computation of $1$-BMF is hard \cite{Gillis:2018}, and can be stated in graphic form as finding a maximum weight biclique of $K_{n,m}$ with edge weights $1$ for $(i,j): x_{ij}=1$ and $-1$ for $(i,j): x_{ij}=0$.


%
%


Many heuristic attempts have been made to approximately compute BMFs by focusing on recursively partitioning the given matrix $X\in\B^{n\times m}$ and computing a $1$-BMF at each step. 
The first such recursive method called Proximus
\cite{Koyuturk:2002:ATA:645806.670310} 
is used to compute BMF under standard arithmetic over the integers. 
For $1$-BMF Proximus uses an alternating iterative heuristic applied to a random starting point which 
is based on the observation that if $\mvec{a}\in \B^n$ is given, then a vector $\mvec{b}\in \B^m$ that minimizes  the distance between $X$ and $\mvec{a}\mvec{b}^\top$ can be computed in $\mathcal{O}(nm)$ time.
%
Since the introduction of Proximus, much research focused on computing efficient and accurate $1$-BMF. 
\cite{Shen:2009:MDP:1557019.1557103} propose an integer program (IP) for $1$-BMF and several linear programming (LP) relaxations of it, one of which leads to a $2$-approximation.
\cite{Shi:2014:6899417} provide a rounding based $2$-approximation for $1$-BMF by using an observation about the vertices of the polytope corresponding to the LP relaxation of an integer program. In \cite{Beckerleg:2020} a modification of the Proximus framework is explored using the approach of \cite{Shen:2009:MDP:1557019.1557103} to compute $1$-BMF at each step.

$k$-BMF under Boolean arithmetic is explicitly introduced in  
\cite{Miettinen:2006:PKDD,Miettinen:2008:DBP:1442800.1442809}, along with a heuristic algorithm called ASSO. The core of ASSO is based on an association rule-mining approach to create matrix $B\in\B^{k\times m}$ and greedily fix $A\in\B^{n\times k}$ with respect to $B$. 
The problem of finding an optimal $A$ with respect to fixed $B$ is NP-hard \cite{Miettinen:2008:PPS:1412758.1412985} but can be solved in $\mathcal{O}(2^k kmn)$ time \cite{Miettinen:2008:DBP:1442800.1442809}.
The association rule-mining approach of \cite{Miettinen:2008:DBP:1442800.1442809} is further improved in \cite{Barahona:2019} by a range of iterative heuristics employing this alternative fixing idea and local search. Another approach based on an alternating style heuristic is explored in \cite{Zhang:2007} to solve a non-linear unconstrained formulation of $k$-BMF with penalty terms in the objective for non-binary entries. \cite{Wan:2020} proposes another iterative heuristic which at every iteration permutes the rows and columns of $X$ to create a dense submatrix in the upper right corner which is used as a rank-1 component in the $k$-BMF.

%
In \cite{Lu:2008:OBM:1546682.1547186,Lu:2014:9400730720140101} an exponential size IP for $k$-BMF is introduced, which uses an explicit enumeration of all possible rows for factor matrix $B$ and corresponding indicator variables. 
To tackle the exponential explosion, 
a heuristic row generation using association rule mining and subset enumeration is developed, but no non-heuristic method is considered.
%
{An exact linear IP for $k$-BMF with polynomially many variables and constraints is presented in \cite{Kovacs:2017}.
	This model uses McCormick envelopes \cite{McCormick:1976} to linearise quadratic terms.}

\subsection{Our Contribution}

In this paper, we present 
a novel IP formulation for $k$-BMF that overcomes several limitations of earlier  approaches. In particular, our formulation does not suffer from permutation symmetry, it does not rely on heuristic pattern mining, and it has a stronger LP relaxation than that of \cite{Kovacs:2017}. 
On the other hand, our new formulation has an exponential number of variables which we tackle using a column generation approach that effectively searches over this exponential space without  explicit enumeration, unlike the complete enumeration used for the exponential size model of \cite{Lu:2008:OBM:1546682.1547186,Lu:2014:9400730720140101}.
Our proposed solution method is able to prove optimality for smaller datasets, while for larger datasets it provides solutions with better accuracy than the state-of-the-art heuristic methods. 
In addition, due to the entry-wise modelling of $k$-BMF in our approach, we can handle matrices with missing entries and our solutions can  be used for binary matrix completion.

The rest of the paper is organised as follows. 
In Section \ref{section_formulation} we briefly discuss the model of \cite{Kovacs:2017} and its limitations. 
In Section \ref{CG} we introduce our integer programming formulation for $k$-BMF, detail a theoretical framework based on the large scale optimisation technique of column generation for its solution and discuss heuristics for the arising pricing subproblems. Finally, in Section \ref{experiments} we demonstrate the practical applicability of our approach on several real world datasets.

\section{Problem Formulation}
\label{section_formulation}

Given a binary matrix $X\in \B^{n\times m}$, and a fixed integer $k\ll \min(n,m)$ we wish to find two binary matrices $A\in \B^{n\times k}$ and $B\in \B^{k \times m}$ to minimise $\|X-A\circ B \|^2_F$, where $\|\cdot\|_F$ denotes the Frobenius norm and $\circ$ stands for Boolean matrix multiplication. 
Since $X$ and $Z:=A\circ B$ are binary matrices, the squared Frobenius and entry-wise $\ell_1$ norm coincide and we can expand the objective function
\begin{equation}
\label{frob_obj}
\|X-Z\|^2_F =
\sum_{(i,j) \in E} (1-z_{ij}) +\sum_{(i,j) \not \in E} z_{ij},
\end{equation}
where $E:=\{(i,j): x_{ij}=1\}$ is the index set of the positive entries of $X$. 
\cite{Kovacs:2017} formulate the problem as an exact integer linear program by introducing variables $y_{i\ell j}$ for the product of $ a_{i\ell}$ and $b_{\ell j}$ $\left( (i,\ell,j)\in \mathcal{F}:= [n] \times [k] \times [m] \right)$, and using McCormick envelopes \cite{McCormick:1976} to avoid the appearance of a quadratic constraint arising from the product. McCormick envelopes represent the product of two binary variables $a$ and $b$ by a new variable $y$ and four linear inequalities given by $MC(a,b) = \{y \in \R: a+b -1 \le y, y\le a, y\le b, 0\le y \}$. The model of \cite{Kovacs:2017} reads as
\begin{align}
(\text{IP}_{\text{exact} })
\;\; 
& \zeta_{IP}= \min_{a,b,y,z} \sum_{(i,j) \in E} (1-z_{ij}) &+\sum_{(i,j) \not \in E} z_{ij}  \label{exact_obj}\\
\text{s.t. }
& y_{i\ell j}\le z_{ij} \le \sum_{\ell=1}^k y_{i\ell j},&
(i,\ell,j) \in \mathcal{F},
\label{sum}\\
& y_{i\ell j}\in MC(a_{i\ell},b_{\ell j}), &(i,\ell,j) \in \mathcal{F},\\
& a_{i\ell},b_{\ell j} \in \B, ~~~ z_{ij} \le 1 , &
(i,\ell,j) \in \mathcal{F}.
\label{McCormick}
\end{align}
The above model is exact in the sense that its optimal solutions correspond to optimal $k$-BMFs of $X$. Most general purpose IP solvers use an enumeration framework, which relies on bounds from the LP relaxation of the IP and consequently, it is easier to solve the IP when its LP bound is tighter.
For $k=1$, we have $y_{i 1 j} = z_{ij}$ for all $i,j$ and the LP relaxation of the model is simply the LP relaxation of the McCormick envelopes which has a rich and well-studied polyhedral structure \cite{Padberg:1989:BQP:70486.70495}.
%
%
However, for $k>1$, $\text{IP}_{\text{exact}}$'s LP relaxation ($\text{LP}_\text{exact}$) only provides a trivial bound.
\begin{lemma}
	\label{lp_0_lemma_text}
	For $k>1$, $\text{LP}_\text{exact}$ has optimal objective value $0$ which is attained by at least $\binom{k}{2}$ solutions.
\end{lemma}
For the proof of Lemma \ref{lp_0_lemma_text}, see Appendix \ref{lp_0}.
Furthermore, for $k>1$ the model is highly symmetric, since $A P \circ P^{-1} B$ is an equivalent solution for any permutation matrix $P$. These properties of the model make it unlikely to be solved to optimality in a reasonable amount of time for a large matrix $X$, though the symmetries can be partially broken by incorporating constraints $\sum_{i}a_{i\ell_1}\geq\sum_{i}a_{i\ell_2}$ for all $\ell_1<\ell_2$. 

Note that constraint \eqref{sum} implies $\frac{1}{k}\sum_{\ell=1}^k y_{i\ell j} \le z_{ij}  \le\sum_{\ell=1}^k y_{i\ell j}$ as a lower and upper bound on each variable $z_{ij}$.
Hence, the objective function may be approximated by 
\begin{equation}
\zeta_{IP}(\rho ) = \sum_{(i,j) \in E} (1-z_{ij} )+\rho \sum_{(i,j) \not \in E} \sum_{\ell=1}^k y_{i\ell j} \label{app_obj},
\end{equation}
where $\rho$ is a parameter of the formulation. By setting $\rho=\frac{1}{k}$ we underestimate the original objective, while setting $\rho=1$ we overestimate. Using \eqref{app_obj} as the objective function reduces the number of variables and constraints in the model. Variables $z_{ij}$ need only be declared for $(i,j)\in E$, and constraint \eqref{sum} simplifies to $ z_{ij} \le\sum_{\ell=1}^k y_{i\ell j} $ for $(i,j)\in E$. 



\vspace{-0.5em}
\section{A Formulation via Column Generation}
\label{CG}
The exact model presented in the previous section relies on polynomially many constraints and variables and constitutes the first approach towards obtaining $k$-BMF 
with optimality guarantees. However, such a compact IP formulation may be weak in the sense that its LP relaxation is a very coarse approximation to the convex hull of integer feasible points and an IP formulation with exponentially many variables or constraints can have the potential to provide a tighter relaxation\cite{Lubbecke:2005}. Motivated by this fact, we introduce a new formulation with an exponential number of variables and detail a column generation framework for its solution.

{Consider enumerating all possible rank-$1$ binary matrices of size $n\times m$ and let 
	\begin{equation*}
	\RR = \{\mvec{a} \mvec{b}^\top: \mvec{a}\in \B^n, \mvec{b}\in \B^m , \mvec{a},\mvec{b} \neq \mvec{0}\} .
	\end{equation*}
	The size of $\RR$ is $|\RR| = (2^n-1)(2^m-1)$ as any pair of binary vectors 
	{\small $\mvec{a},\mvec{b}\neq \mvec{0}$}  leads to a unique rank-1 matrix  {\small $Y=\mvec{a}\mvec{b}^\top$} with 
	{\small $Y_{ij}=1$} for {\small $\{ (i,j): a_i=1, b_j=1\}$}.
	%
	Define a binary decision variable $q_\ell$ to denote if the $\ell$-th rank-$1$ binary matrix in $ \RR$ is included in a rank-$k$ factorisation of $X$ ($q_\ell=1$), or not  ($q_\ell=0$).
}
Let $\mvec{q} \in \B^{|\RR|}$ be a vector that has a component $q_\ell$ for each matrix in $\RR$.
We form a $\B$-matrix $M$ of dimension  $ nm \times |\RR|$ whose rows correspond to entries of an $n\times m$ matrix, columns to rank-$1$ binary matrices in $\RR$  and $M_{(i,j),\ell}=1$ if the $(i,j)$-th entry of the $\ell$-th rank-$1$ binary matrix in $ \RR$ is $1$, $M_{(i,j),\ell}=0$ otherwise.
We split $M$ horizontally into two matrices $M_{\mvec{0}}$ and $M_{\mvec{1}}$, so that rows of $M$ corresponding to a positive entry of the given matrix $X$  are in $M_{\mvec{1}}$ and the rest of rows of $M$ in $M_{\mvec{0}}$,
\begin{equation}
\label{partition}
M = \begin{bmatrix}
M_{\mvec{0}} \\
M_{\mvec{1}} 
\end{bmatrix}
\;\text{where} 
\begin{array}{l}
M_{\mvec{0}}  \in \B^{(nm-|E|) \times |\RR|}, \\
M_{\mvec{1}}  \in \B^{|E|\times |\RR|}.
\end{array}
\end{equation}
The following Master Integer Program over an exponential number of variables is an exact model for $k$-BMF,
\begin{align}
(\text{MIP}_{\text{exact}}) \; \zeta_{\text{MIP}} = \min \;& \mvec{1}^\top \mvec{\xi}  +  \mvec{1}^\top\mvec{\pi}  \\
\text{s.t. }
& M_{\mvec{1}} \mvec{q} +\mvec{\xi}\ge \mvec{1}  \label{ones}\\
& M_{\mvec{0}} \mvec{q}  \le k \mvec{\pi}\ \label{zeros}\\
&\mvec{1}^\top \mvec{q} \le k & \label{k_rectangles}\\
& \mvec{\xi} \ge \mvec{0}, \mvec{\pi} \in \B^{nm-|E|} ,\\
& \mvec{q} \in \B^{|\RR|}. \label{integral_mip_exact}
\end{align}
Constraint \eqref{k_rectangles} ensures that at most $k$ rank-$1$ matrices are active in a factorisation. 
Variables $\xi_{ij}$ correspond to positive entries of $X$, and are forced by constraint \eqref{ones}  to take value $1$ and increase the objective if the $(i,j)$-th positive entry of $X$ is not covered. 
Similarly, variables $\pi_{ij}$ correspond to zero entries of $X$ and are forced to take value $1$ by constraint \eqref{zeros} if the $(i,j)$-th zero entry of $X$ is erroneously covered in a factorisation. 
One of the imminent advantages of $\text{MIP}_{\text{exact}}$ is using indicator variables directly for rank-1 matrices instead of the entries of factor matrices $A,B$, hence no permutation symmetry arises. In addition, for all $k$ not exceeding a certain number that depends on $X$, the LP relaxation of $\text{MIP}_{\text{exact}}$ ($\text{MLP}_{\text{exact}}$) has strictly positive optimal objective value.
\begin{lemma} 
	\label{lp_NOT_zero_text} Let $i(X)$ be the isolation number of $X$.
	For all $k<i(X)$, we have $0<\zeta_\text{MLP}  $.
\end{lemma}
For the definition of isolation number and the proof of Lemma \ref{lp_NOT_zero_text} see Appendix \ref{bound_on_lp_exact}.
Similarly to the polynomial size exact model $\text{IP}_{\text{exact}}$ in the previous section, we consider a modification of $\text{MIP}_{\text{exact}}$ with an objective that is analogous to the one in Equation \eqref{app_obj},
\begin{align}
(\text{MIP}(\rho)) \; z_{\text{MIP}(\rho)}=\min \;& \mvec{1}^\top \mvec{\xi}  +  \rho \, \mvec{1}^\top M_{\mvec{0}} \mvec{q} \label{hamming_obj}\\
\text{s.t. } & \eqref{ones},\; \eqref{k_rectangles} \text{ hold and } \\
& \mvec{\xi} \ge \mvec{0}, ~~~\mvec{q} \in \B^{|\RR|}. \nonumber
\end{align}
The objective of MIP($\rho$) simply counts the number of positive entries of $X$ that are not covered by any of the $k$ rank-1 matrices chosen, plus the number of times zero valued entries of $X$ are erroneously covered weighted by  parameter $\rho$. Depending on the selection of parameter $\rho$, $\text{MIP}(\rho)$ provides a lower or upper bound on $\text{MIP}_{\text{exact}}$.
We denote the LP relaxation of $\text{MIP}(\rho)$ by MLP($\rho$).
\begin{lemma}
	\label{lemma_coincide}
	For $\rho=\frac{1}{k}$, the optimal objective values of the LP relaxations $\text{MLP}_{\text{exact}}$ and $\text{MLP}(\frac{1}{k})$ coincide. 
\end{lemma}
For a short proof of Lemma \ref{lemma_coincide} see Appendix \ref{subsection_coincide}.
Combining Lemmas \ref{lp_0_lemma_text}, \ref{lp_NOT_zero_text} and \ref{lemma_coincide} we obtain the following relations between formulations $\text{IP}_{\text{exact}}$, $\text{MIP}_{\text{exact}}$, MIP($\rho$) and their LP relaxations for $k>1$,
\begin{align}
&z_{\text{MIP}(\frac{1}{k})} \le \zeta_{IP}=\zeta_{\text{MIP}} \le z_{\text{MIP}(1)}, \\
&0=\zeta_{LP} \le z_{\text{MLP}(\frac{1}{k})} = \zeta_\text{MLP} \le z_{\text{MLP}(1)}.
\end{align}
Let $\mvec{p}$ be the dual variable vector associated to constraints \eqref{ones} and 
$\mu$ be the dual variable to constraint \eqref{k_rectangles}.
%
Then the dual of $\text{MLP}(\rho)$ is given by
\begin{align}
(\text{MDP}(\rho) )\; & z_{\text{MDP}(\rho)} =\max \mvec{1}^\top \mvec{p}  - k \mu \\
\text{s.t. }
& M_{\mvec{1}}^\top \mvec{p} -\mu \mvec{1} \le   \rho\, M_{\mvec{0}}^\top\mvec{1}, \\
&\mu \ge 0,\; \mvec{p} \in [0,1]^{|E|}.
\end{align}
Due to the number of variables in the formulation, it is not practical to solve $\text{MIP}(\rho)$ or 
its LP relaxation $\text{MLP}(\rho)$ explicitly. \textit{Column generation} (CG) is a technique to solve large LPs by iteratively generating only the variables which have the potential to improve the objective function \cite{Nemhauser:1998}. 
The CG procedure is initialised by explicitly solving  a Restricted Master LP which has a small subset of the variables in MLP($\rho$). 
The next step is to identify a missing variable with a \textit{negative reduced cost} to be added to this Restricted MLP($\rho$). To  avoid explicitly considering all missing variables, a \textit{pricing problem} is formulated and 
solved. The solution of the pricing problem either returns a variable with negative reduced cost and the procedure is iterated; or proves that no such variable exists and hence the solution of the Restricted MLP($\rho$) is optimal for the complete formulation MLP($\rho$).

We use CG to solve MLP($\rho$) by considering a sequence $(t = 1,2,...)$ of  Restricted MLP($\rho$)'s with constraint matrix $M^{(t)}$ being a subset of columns of $M$, where each column $\mvec{y}\in \B^{nm}$ of $M$ corresponds to a flattened rank-$1$ binary matrix $\mvec{a}\mvec{b}^\top$ according to Equation \eqref{partition}.
The constraint matrix of the first Restricted MLP($\rho$) may be left empty or can be warm started by identifying a few rank-$1$ matrices in $\RR$, say from a heuristic solution.
Upon successful solution of the $t$-th Restricted MLP($\rho$), we obtain a vector of dual variables $[\mvec{p}^*,\mu^*]\ge \mvec{0}$ optimal for the $t$-th Restricted MLP($\rho$). 
To identify a missing column of $M$ that has a negative reduced cost, we solve the following pricing problem (PP): 
\begin{align*}
(\text{PP}) & \;\omega( \mvec{p}^*)=\max_{a,b,y}  \sum_{(i,j)\in E} p_{ij}^{*} y_{ij} -  \rho \sum_{(i,j)\not\in E} y_{ij} &\\
\text{s.t. }
&y_{ij} \in MC(a_i, b_j),\;
a_{i},b_{j}\in\B,\;i\in [n], j\in [m].
\end{align*}
The objective of PP depends on the current dual solution $[\mvec{p}^*,\mu^*]$ 
and its optimal solution corresponds to a rank-$1$ binary matrix $\mvec{a}\mvec{b}^\top$ whose corresponding variable $q_\ell$ in MLP($\rho$) has the smallest reduced cost. 
If $\omega( \mvec{p}^*)\le\mu^*$, then the dual variables $[\mvec{p}^*, \mu^*]$ of the Restricted MLP($\rho$) are feasible for MDP($\rho$) and hence the current solution of the Restricted MLP($\rho$) is optimal for the full formulation MLP($\rho$).
If $\omega( \mvec{p}^*)> \mu^*$, then the variable $q_\ell$ associated with the rank-$1$ binary matrix $\mvec{a}\mvec{b}^\top$ is added to  the  Restricted MLP($\rho$) 
and the procedure is iterated.
CG optimally terminates if at some iteration we have $\omega( \mvec{p}^*)\le  \mu^*$.

{To apply the CG approach above 
	to $\text{MLP}_{\text{exact}}$ only a small modification needs to be made.} 
The Restricted $\text{MLP}_{\text{exact}}$ provides dual variables for constraints \eqref{zeros} which are used in the objective of PP for coefficients of $y_{ij}$ $(i,j)\not \in E$. 
Note however, that CG cannot be used to solve a modification of $\text{MLP}_{\text{exact}}$ in which constraints \eqref{zeros} are replaced by exponentially many constraints $(M_{\mvec{0}})_{\ell} \; q_\ell \le \mvec{\pi}$ for $\ell\in[|\RR|]$ where $(M_{\mvec{0}})_{\ell}$ denotes the $\ell$-th column of $M_{\mvec{0}}$, see Appendix \ref{strong_lp_sec}.

{If the optimal solution of MLP($\rho$) is integral, then it is also optimal for $\text{MIP}(\rho)$. However, if it is fractional, then it only provides a lower bound on the optimal value of $\text{MIP}(\rho)$.
	In this case we obtain an integer feasible solution by simply adding integrality constraints on the  variables of the final Restricted MLP($\rho$) and solving it as an integer program. 
	If $\rho=1$, the optimal solution of this integer program is optimal for $\text{MIP}(1)$ if the objective of the Restricted $\text{MIP}(1)$ and the ceiling of the Restricted MLP(1) agree.
}
To solve the MIP($\rho$) to optimality in all cases, one needs to embed CG into branch-and-bound 
which we do not do. However, note that even if the CG procedure is terminated prematurely, one can still obtain a lower bound on MLP($\rho$) and $\text{MIP}(\rho)$ as follows. Let the objective value of any of the Restricted MLP($\rho$)'s be
\begin{equation}
z_{\text{RMLP}}=\mvec{1}^\top \mvec{\xi}^* + \rho \mvec{1}^\top M^{*}_0\mvec{q}^* = \mvec{1}^\top \mvec{p}^* - k \mu^*
\end{equation}
where 
$[\mvec{\xi}^*, \mvec{q}^*]$ is the solution of the Restricted MLP($\rho$), $[\mvec{p}^*, \mu^*]$ is the solution of the dual of the Restricted MLP($\rho$) and $\mvec{1}^\top M^*_0$ is the objective coefficient of columns in the Restricted MLP($\rho$). Assume that we solve PP to optimality and we obtain a column $\mvec{y}$ for which the reduced cost is negative, $\omega( \mvec{p}^*)> \mu^* $. In this case, we can construct a feasible solution to MDP($\rho$) by setting $\mvec{p}:=\mvec{p}^*$ and $\mu := \omega( \mvec{p}^*)$ and get the following bound on the optimal value $z_{\text{MLP}(\rho)}$ of MLP($\rho$),
\begin{equation}
\label{guarantee}
z_{\text{MLP}(\rho)} \ge \mvec{1}^\top \mvec{p}^* - k\, \omega( \mvec{p}^*)= z_{\text{RMLP}} - k (\omega( \mvec{p}^*)- \mu^*).
\end{equation}
If we do not have the optimal solution to PP but have an upper bound $\bar \omega( \mvec{p}^*)$ on it, $\omega( \mvec{p}^*)$ can be replaced by $\bar \omega( \mvec{p}^*)$ in equation \eqref{guarantee} and the bound on MLP($\rho$) still holds. Furthermore, this lower bound on MLP($\rho$) provides a valid lower bound on $\text{MIP}(\rho)$. Consequently, our approach always produces a bound on the optimality gap of the final solution which heuristic methods cannot do. We have, however, no a priory (theoretical) bound on this gap.

\subsection{The Pricing Problem}

The efficiency of the CG procedure described above greatly depends on solving PP  efficiently. In standard form PP can be written as a bipartite binary quadratic program (BBQP) 
\begin{equation}
\label{BBQ}
(\text{PP}) \quad \omega(\mvec{p}^*) = \max_{\mvec{a}\in \B^n, \mvec{b} \in \B^m} \mvec{a} ^\top H \mvec{b}
\end{equation}
for $H$ an $n\times m$ matrix with $h_{ij} = p^*_{ij}\in [0,1]$ for $(i,j)\in E$ and $h_{ij} = -\rho$ for $(i,j)\not \in E$. BBQP is NP-hard in general as it includes the maximum edge biclique problem \cite{Peeters:2003}, hence for large $X$ it may take too long to solve PP to optimality at each iteration. 
To speed up computations, the IP formulation of PP  
may be improved by eliminating redundant constraints. The McCormick envelopes 
set two lower and two upper bounds on $y_{ij}$. Due to the objective function it is possible to declare the lower (upper) bounds  $y_{ij}$ for only $(i,j)\not \in E$ ($(i,j)\in E$) without changing the optimum. 



If a heuristic approach to PP  provides a solution with negative reduced cost, 
then it is valid to add this heuristic solution as a column to the next Restricted MLP($\rho$).
Most heuristic algorithms that are available for BBQP build on the idea that the optimal $\mvec{a}\in \B^n$ with respect to  a fixed $\mvec{b}\in \B^m$ can be computed in $\mathcal{O}(nm)$ time and this procedure can be iterated by alternatively fixing $\mvec{a}$ and $\mvec{b}$. 
\cite{Punnen:2012} present several local search heuristics for BBQP along with a simple greedy algorithm.
Below we detail this greedy algorithm 
and introduce some variants of it which we use in the next section to provide a warm start to PP at every iteration of the CG procedure. 
%


%

The greedy algorithm of \cite{Punnen:2012} aims to set entries of $\mvec{a}$ and $\mvec{b}$ to $1$ which correspond to rows and columns of $H$ with the largest positive weights. 
In the first phase  of the algorithm, the row indices $i$ of $H$ are put in decreasing order according to their sum of positive entries, so $\gamma^+_{i} \ge \gamma^+_{i+1}$ where $ \gamma^+_i:=\sum_{j=1}^m \max(0,h_{ij})$.  
Then sequentially according to this ordering, $a_i$ is set to $1$ if $\sum_{j=1}^m \max ( 0, \sum_{\ell=1}^{i-1}  a_\ell h_{\ell j} ) < \sum_{j=1}^m \max ( 0, \sum_{\ell=1}^{i-1} a_\ell h_{\ell j} + h_{ij})$ and $0$ otherwise. In the second phase, $b_j$ is set to $1$ if $(\mvec{a}^\top H)_j >0$, $0$ otherwise. The precise outline of the algorithm is given in Appendix \ref{greedy_punnen}.


%


There are many variants of the greedy algorithm one can explore.
First, the solution greatly depends on the ordering of $i$'s in the first phase. If for some $i_1\not = i_2$ we have $\gamma^{+}_{i_1} = \gamma^{+}_{i_2}$, comparing the sum of negative entries of rows $i_1$ and $i_2$ can put more ``influential'' rows of $H$ ahead in the ordering. Let us call this ordering the \textit{revised ordering} and the one which only compares the positive sums as the \textit{original ordering}.
Another option is to use a completely \textit{random order} of $i$'s or to apply a small perturbation to sums $\gamma^+_i$ to get a \textit{perturbed} version of the revised or original ordering. None of the above ordering strategies clearly dominates the others in all cases but they are fast to compute hence one can evaluate all five ordering strategies (original, revised, original perturbed, revised perturbed, random) and pick the best one.
Second, the algorithm as presented above first fixes $\mvec{a}$ and then $\mvec{b}$. Changing the order of fixing $\mvec{a}$ and $\mvec{b}$ can yield a different result hence it is best to try for both $H$ and $H^\top$. In general, it is recommended to start the first phase on the smaller dimension. 
Third, the solution from the greedy algorithm may be improved by computing the optimal $\mvec{a}$ with respect to fixed $\mvec{b}$. This idea then can be used to fix $\mvec{a}$ and $\mvec{b}$ in an alternating fashion and stop when no changes occur in either.


\section{Experiments}
\label{experiments}

The CG approach introduced in the previous section provides a theoretical framework for computing $k$-BMF with optimality guarantees. In this section we present some experimental results with CG to demonstrate the practical applicability of our approach on eight real world categorical datasets that were downloaded from online repositories \cite{UCI}, \cite{BOOKS}.
Table \ref{dimensions} shows a short summary of the eight datasets used, details on converting categorical columns into binary and missing value treatment can be found in Appendix \ref{data}. Table \ref{dimensions} also shows the value of the isolation number $i(X)$
for each dataset, which provides a lower bound on the Boolean rank \cite{Monson:1995}.

\begin{table}[h]
	\centering
	\setlength{\tabcolsep}{3pt}
	\begin{tabular}{lrrrrrrrr}
		\toprule
		&zoo & tumor 
		& hepat
		&heart& lymp & audio& apb& votes\\
		\midrule
		n     & 101   & 339 & 155    & 242   & 148   & 226   & 105   & 434 \\
		m     & 17    & 24     & 38  & 22    & 44    & 94    & 105   & 32 \\
		$i(X)$ & 	16& 24   & 30& 22  & 42  & 90 & 104 &  $30$\\
		\%1s  & 44.3  & 24.3 & 47.2  & 34.4 & 29.0  & 11.3  & 8.0  & 47.3 \\
		\bottomrule
	\end{tabular} %
	\caption{\label{dimensions}Summary of binary real world datasets}
\end{table}

\begin{table}[b!]
	\centering
	\setlength{\tabcolsep}{3.5pt}
	\begin{tabular}{lrrrrrr}
		\toprule
		&    & \multicolumn{1}{c}{zoo} &       &       & \multicolumn{1}{c}{tumor} &  \\
		k     & \multicolumn{1}{c}{MIP(1)} & KGH17 & LVA08 & \multicolumn{1}{c}{MIP(1)} & KGH17 & LVA08 \\
		\midrule
		2     & 0.0   & 0.0   & 100.0 & 0.9   & 40.8  & * \\
		5     & 0.0   & 59.2  & 100.0 & 9.3   & 98.0  & * \\
		10    & 3.0   & 95.8 & 100.0 & 28.4  & 100.0 & * \\
		\bottomrule
	\end{tabular}%
	\caption{\label{gap}\% optimality gap after 20 mins under objective \eqref{app_obj} }
\end{table}%

\begin{table*}[t]
	\centering
	\setlength{\tabcolsep}{3pt} 
	\begin{tabular}{cl l l l l l l l l }
		\toprule
		$k$  &       &\multicolumn{1}{l}{zoo} & \multicolumn{1}{l}{tumor} & \multicolumn{1}{l}{hepat} & \multicolumn{1}{l}{heart} & \multicolumn{1}{l}{lymp} & \multicolumn{1}{l}{audio} & \multicolumn{1}{l}{apb} & \multicolumn{1}{l}{votes} \\
		\midrule
		\multirow{3}[2]{*}{2} & MLP($\frac{1}{k}$) & 206.5 & 1178.9 & 978.7 & 882.9 & 917.2 & 1256.5 & 709   & 1953 \\
		& MLP(1) & 272   & 1409.8 & 1384  & 1185  & 1188.8 & 1499  & 776   & 2926 \\
		& MIP(1) & 272(271)   & 1411  & 1384(1382) & 1185  & 1197(1184) & 1499  & 776   & 2926 \\
		\midrule
		\multirow{3}[2]{*}{5} & MLP($\frac{1}{k}$) & 42.8 & 463.9 & 333.1 & 291.0 & 366.7 & 654.2 & 433.5 & 715.5 \\
		& MLP(1) & 127   & 1019.3 & 1041.1 & 736   & 914.0 & 1159.3 & 683.0 & 2135.5 \\
		& MIP(1) & 127(125) & 1029  & 1228  & 736   & 997(991) & 1176  & 684(683) & 2277(2274) \\
		\midrule
		\multirow{3}[2]{*}{10} & MLP($\frac{1}{k}$) & 4.8  & 192.8 & 142.5 & 102.3 & 165.1 & 351.4 & 166.8 & 307.9 \\
		& MLP(1) & 38.8 & 575.5 & 734.8 & 419   & 653.2 & 867.2 & 574.2 & 1409.5 \\
		& MIP(1) & 40    & 579   & 910   & 419   & 737(732) & 893   & 577(572) & 1566(1549) \\
		\bottomrule
	\end{tabular}%
	\caption{\label{formulations} Primal objective values of MLP(1), MLP($\frac{1}{k}$), MIP(1) after 20 mins of CG}
\end{table*}%

\begin{table*}[t]
	\centering
	\begin{tabular}{clrrrrrrrr}
		\toprule
		&       & \multicolumn{1}{l}{zoo} & \multicolumn{1}{l}{tumor} & \multicolumn{1}{l}{hepat} & \multicolumn{1}{l}{heart} & \multicolumn{1}{l}{lymp} & \multicolumn{1}{l}{audio} & \multicolumn{1}{l}{apb} & \multicolumn{1}{l}{votes} \\
		\midrule
		\multirow{8}[2]{*}{k=2} & CG    & \textbf{271} & 1411  & \textbf{1382} & \textbf{1185} & 1184  & \textbf{1499} & \textbf{776} & \textbf{2926} \\
		& $\text{IP}_\text{exact}$ & \textbf{271} & \textbf{1408} & 1391  & 1187  & \textbf{1180} & \textbf{1499} & \textbf{776} & \textbf{2926} \\
		& ASSO++ & 276   & 1437  & 1397  & 1187  & 1202  & 1503  & \textbf{776} & \textbf{2926} \\
		& $k$-greedy & 325   & 1422  & 1483  & 1204  & 1201  & \textbf{1499} & \textbf{776} & 2929 \\
		& pymf  & 276   & 1472  & 1418  & 1241  & 1228  & 1510  & 794   & 2975 \\
		& ASSO  & 367   & 1465  & 1724  & 1251  & 1352  & 1505  & 778   & 2946 \\
		& NMF   & 291   & 1626  & 1596  & 1254  & 1366  & 2253  & 809   & 3069 \\
		& MEBF & 348 &	1487	&1599	&1289	&1401	&1779&	812	&3268\\
		\midrule
		\multirow{8}[2]{*}{k=5} & CG    & \textbf{125} & \textbf{1029} & \textbf{1228} & \textbf{736} & \textbf{991} & \textbf{1176} & \textbf{683} & \textbf{2272} \\
		&$\text{IP}_\text{exact}$& 133   & 1055  & \textbf{1228} & 738   & 1029  & 1211  & 690   & 2293 \\
		& ASSO++ & 133   & 1055  & \textbf{1228} & 738   & 1039  & 1211  & 694   & 2302 \\
		& $k$-greedy & 233   & 1055  & 1306  & 748   & 1063  & 1211  & 690   & 2310 \\
		& pymf  & 142   & 1126  & 1301  & 835   & 1062  & 1245  & 730   & 2517 \\
		& ASSO  & 354   & 1092  & 1724  & 887   & 1352  & 1505  & 719   & 2503 \\
		& NMF   & 163   & 1207  & 1337  & 995   & 1158  & 1565  & 762   & 2526 \\
		& MEBF & 173 &	1245	&1439	&929&	1245	&1672	&730&	2832\\
		\midrule
		\multirow{8}[2]{*}{k=10} & CG    & \textbf{40} & \textbf{579} & 910   & \textbf{419} & \textbf{730} & \textbf{893} & \textbf{572} & \textbf{1527} \\
		& $\text{IP}_\text{exact}$ & 41    & 583   & \textbf{902} & \textbf{419} & 805   & 919   & 590   & 1573 \\
		& ASSO++ & 55    & 583   & 910   & \textbf{419} & 812   & 922   & 591   & 1573 \\
		& $k$-greedy & 184   & 675   & 1088  & 565   & 819   & 976   & 611   & 1897 \\
		& pymf  & 96    & 703   & 1186  & 581   & 987   & 1106  & 602   & 2389 \\
		& ASSO  & 354   & 587   & 1724  & 694   & 1352  & 1505  & 661   & 2503 \\
		& NMF   & 153   & 826   & 1337  & 995   & 1143  & 1407  & 689   & 2481 \\
		& MEBF & 122 &	990	 &1328	&777&	1004	&1450	&662&	2460\\
		\bottomrule
	\end{tabular}%
	\caption{\label{COMP}Factorisation errors in $\|\cdot\|_F^2$ for eight methods for $k$-BMF}
\end{table*}%

Since the efficiency of CG greatly depends on the speed of generating columns, let us illustrate the speed-up gained by using heuristic pricing. 
At each iteration of CG, 30 variants of the greedy heuristic are computed to obtain an initial feasible solution to PP. The 30 variants of the greedy algorithm use the original and revised ordering, their transpose and perturbed version and 22 random orderings.
All greedy solutions are improved by the alternating heuristic until no further improvement is found.
Under \textit{exact} pricing, the best heuristic solution is used as a warm start and PP is solved to optimality at each iteration using \cite{CPLEXmanual}.
In simple heuristic (\textit{heur}) pricing, if the best heuristic solution to PP has negative reduced cost, $\omega_{\text{heur}}(\mvec{p}^*)>\mu^*$, then the heuristic column is added to the next Restricted MLP($\rho$). If at some iteration, the best heuristic column does not have negative reduced cost, CPLEX is used to solve PP to optimality for that iteration.
The multiple heuristic (\textit{heur\_multi}) pricing is a slight modification of the simple heuristic strategy, in which at each iteration all columns with negative reduced cost are added to the next Restricted MLP($\rho$). 

\begin{figure}[b!]
	\centering
	\includegraphics[scale=0.8]{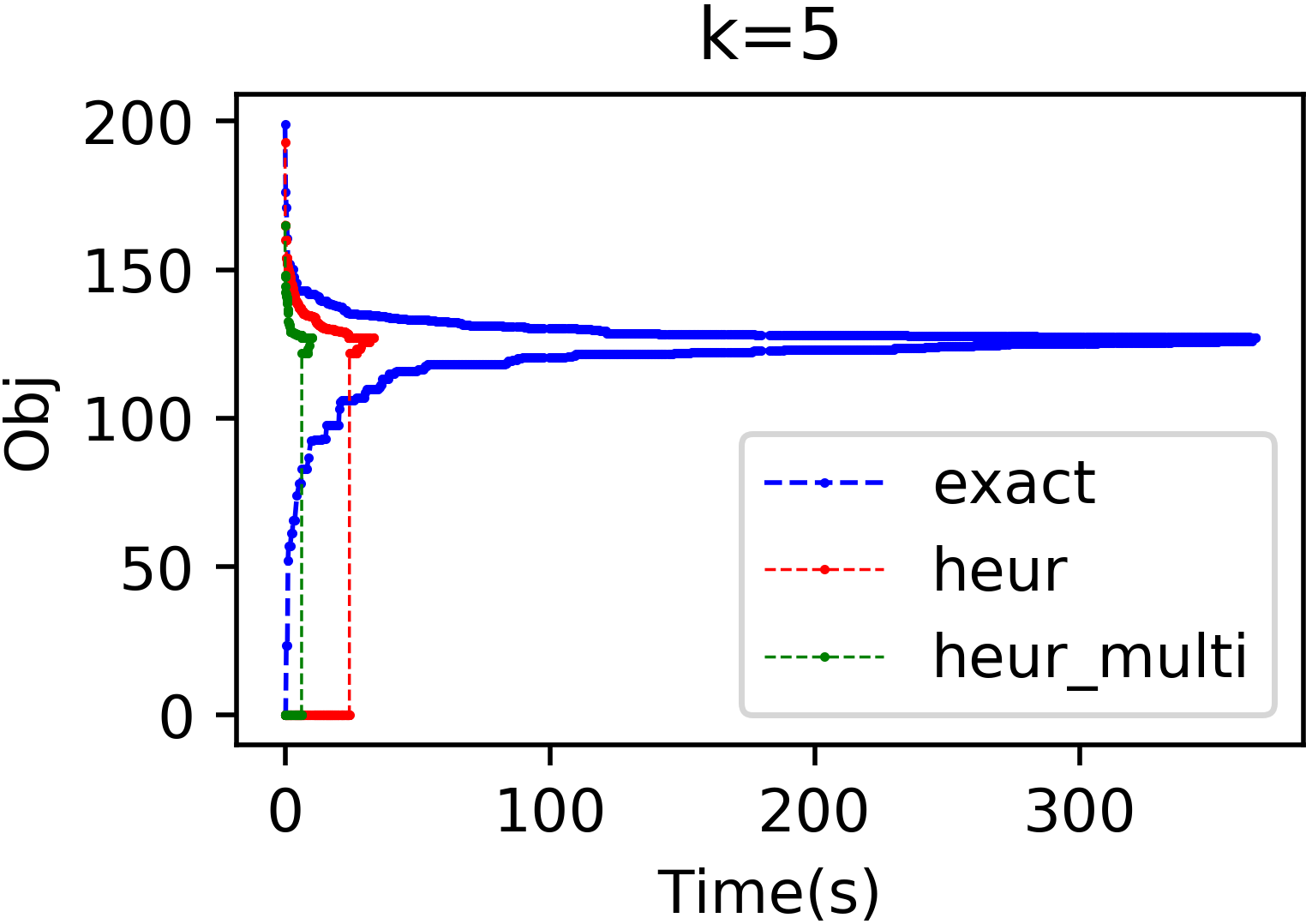}\\
	\includegraphics[scale=0.8]{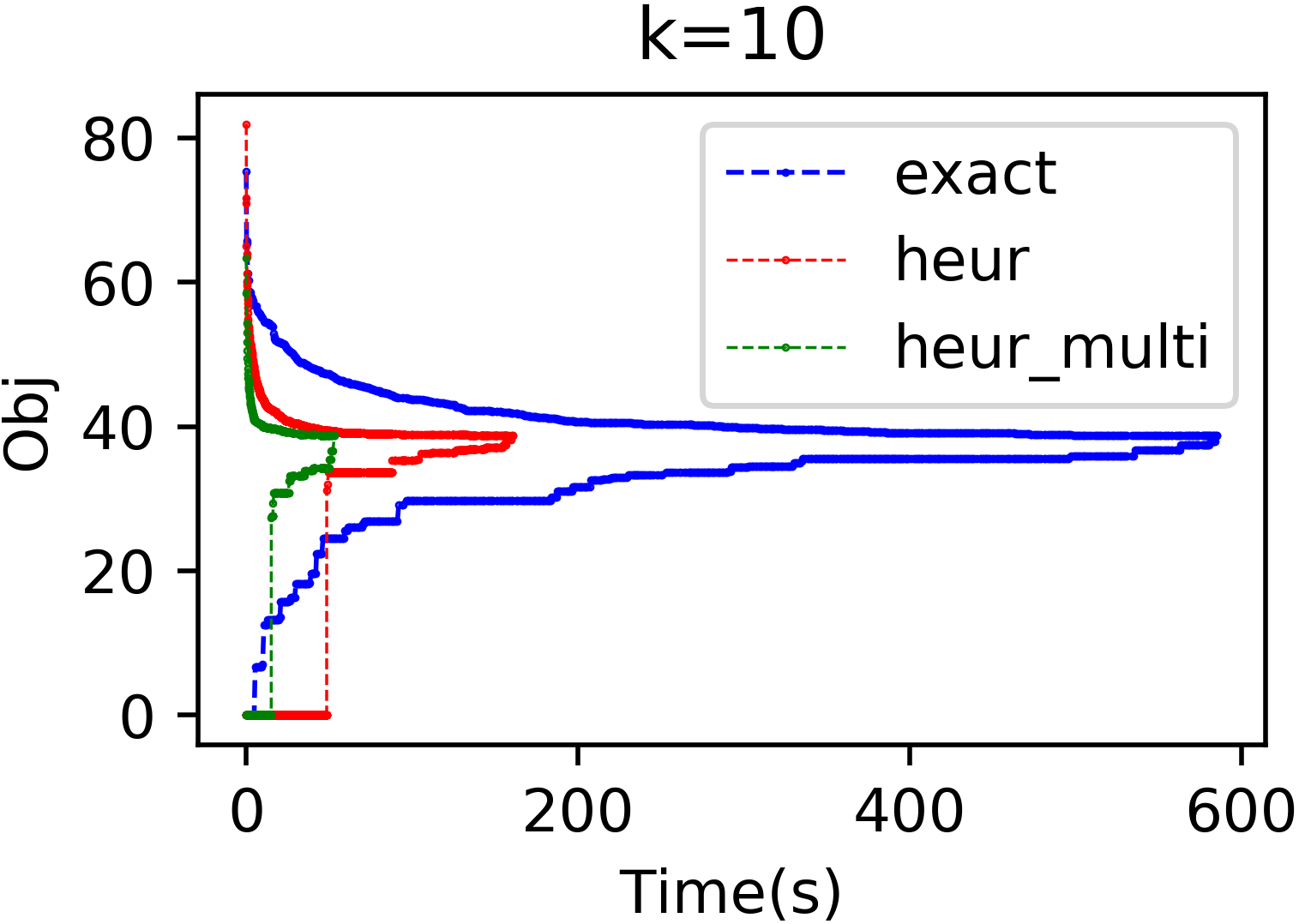}
	\caption{\label{zoo_pricing}Comparison of pricing strategies for solving $\text{MLP}(1)$ on the zoo dataset }
\end{figure}
Figure \ref{zoo_pricing} indicates the differences between pricing strategies when solving MLP(1) via CG for $k=5,10$ on the zoo dataset. The primal objective value of MLP(1) (decreasing curve) and the value of the dual bound (increasing curve) computed using the formula in equation \eqref{guarantee} are plotted against time. 
Sharp increases in the dual bound for heuristic pricing strategies correspond to iterations in which CPLEX was used to solve PP, as for the evaluation of the dual bound on MLP(1) we need a strong upper bound on $\omega(\mvec{p}^*)$ which heuristic solutions do not provide. 
While we observe a tailing off effect \cite{Lubbecke:2005} on all three curves, both heuristic pricing strategies provide a significant speed-up from exact pricing, with adding multiple columns at each iteration being the fastest.

In Table \ref{gap} we present computational results comparing the optimality gap ($100\times\tfrac{\text{best integer}-\text{best bound}}{\text{best integer}}$) of MIP(1), the compact formulation of \cite{Kovacs:2017} (\textit{KGH17}) and the exponential formulation of \cite{Lu:2008:OBM:1546682.1547186} (\textit{LVA08}) under objective \eqref{app_obj} using a 20 mins time budget.
See Appendix \ref{app_table_2_formulations} for the precise statement of formulations KGH17 and LVA08  under objective \eqref{app_obj}.
Reading in the full exponential size model LVA08  using 16 GB memory is not possible for datasets other  than zoo.
Table \ref{gap} shows that different formulations and algorithms to solve them make a difference in practice: our novel exponential formulation MIP(1) combined with an effective computational optimization approach (CG) produces solutions with smaller optimality gap than the compact formulation as it scales better and it has a stronger LP relaxation.

In order for CG to terminate with a certificate of optimality, at least one pricing problem has to be solved to optimality. Unfortunately for the larger datasets this cannot be achieved in under 20 mins. Therefore, for datasets other than zoo, we change the multiple heuristic pricing strategy as follows: 
We impose an overall time limit of 20 mins on the CG process and use the barrier method in CPLEX as the LP solver for the Restricted MLP($\rho$) at each iteration. 
In order to maximise the diversity of columns added at each iteration, we choose at most two columns with negative reduced cost that are closest to being mutually orthogonal. 
If CPLEX has to be used to improve the heuristic pricing solution, we do not solve PP to optimality 
but abort CPLEX if a column with negative reduced cost has been found. While these modifications result in a speed-up, they reduce the chance of obtaining a strong dual bound. In case a strong dual bound is desired, 
we may continue applying CG iterations with exact pricing after the 20 mins of heuristic pricing have run their 
course.



In our next experiment, we explore the differences between formulations MLP($\frac{1}{k}$), MLP(1) and MIP(1). We warm start CG by identifying a few heuristic columns using the code of \cite{Barahona:2019} and a new fast heuristic ($k$-\textit{greedy}) which sequentially computes $k$ rank-1 binary matrices via the greedy algorithm for BBQP starting with the coefficient matrix $H=2X-1$ and then setting entries of $H$ to zero that correspond to entries of $X$ that are covered. For the precise outline of $k$-greedy, see Appendix \ref{rank-k greedy}.

Table \ref{formulations} shows the primal objective values of MLP(1) and MLP($\frac{1}{k}$) with heuristic pricing using a time limit of 20 mins, and the objective value of MIP(1) solved on the columns generated by MLP(1). If the error measured in $\|\cdot\|^2_F$ differs from the objective of MIP(1), the former is shown in parenthesis. It is interesting to observe that MLP(1) has a tendency to produce near integral solutions and that the objective value of MIP(1) often coincides with the error measured in $\|\cdot\|^2_F$.  We note that once a master LP formulation is used to generate columns, any of the MIP models could be used to obtain an integer feasible solution. In experiments, we found that formulation MIP($\rho$) is solved much faster than $\text{MIP}_{\text{exact}}$ and that setting $\rho$ to $1$ or $0.95$ provides the best integer solutions.


We compare the CG approach against the most widely used $k$-BMF heuristics and the exact model $\text{IP}_{\text{exact}}$.
The heuristic algorithms we evaluate include 
the ASSO algorithm \cite{Miettinen:2006:PKDD, Miettinen:2008:DBP:1442800.1442809}, 
the alternating iterative local search algorithm (ASSO++) of \cite{Barahona:2019}  which uses ASSO as a starting point, 
algorithm $k$-greedy detailed in Appendix \ref{rank-k greedy}, 
the penalty objective formulation (pymf) of \cite{Zhang:2007} via the implementation of \cite{Schinnerl:2017}
and the permutation-based heuristic (MEBF) \cite{Wan:2020}.
We also evaluate $\text{IP}_{\text{exact}}$ with a time limit of $20$ mins and provide the heuristic solutions of ASSO++ and $k$-greedy as a warm start to it. 
In addition, we compute rank-$k$ NMF and binarise it with a threshold of $0.5$. 
The exact details and parameters used in the computations can be found in Appendix \ref{app_comp}.\footnote{Appendix is available at \texttt{arxiv.org/abs/2011.04457}.} 
Our CG approach (CG) results are obtained by generating columns for 20 mins using formulation MLP(1) with a warm start of initial columns obtained from  ASSO++ and $k$-greedy, then solving MIP($\rho$) for $\rho$ set to $1$ and $0.95$ over the generated columns and picking the best.
Table \ref{COMP} shows the factorisation error in $\|\cdot\|_F^2$ after evaluating the above described methods on all datasets for $k=2,5,10$. The best result for each instance is indicated in boldface. We observe that CG provides the strictly smallest error for 15 out of 24 cases.

\section{Conclusion}

In this paper, we studied the rank-$k$ binary matrix factorisation problem under Boolean arithmetic. We introduced a new integer programming formulation and detailed a method using column generation for its solution. Our experiments indicate that our method using 20 mins time budget is producing more accurate solutions than most heuristics available in the literature and is able to prove optimality for smaller datasets.
In certain critical applications such as medicine, spending 20 minutes to obtain a higher accuracy factorisation  with a bound on the optimality gap 
can be easily justified.
In addition, solving BMF to near optimality via our proposed method paves the way to more robustly benchmark  heuristics for $k$-BMF.
Future directions that could be explored are related to designing more accurate heuristics and faster exact algorithms for the pricing problem. In addition, a full branch-and-price algorithm implementation would be beneficial once the pricing problems are solved more efficiently.

\section*{Acknowledgements}
During the completion of this work R.\'A.K was supported by The Alan Turing Institute and The Office for National Statistics.

\bibliography{refs}

\include{appendix}

\end{document}

%% file: appendix.tex
\section*{Appendix}

\subsection{The strength of the LP relaxation of $\text{IP}_\text{exact}$}
\label{lp_0}

\setcounter{theorem}{0}
\begin{lemma}
	\label{lp_0_lemma}
	For $k>1$, the LP relaxation of $\text{IP}_\text{exact}$ ($\text{LP}_\text{exact}$) has optimal objective value $0$ which is attained by at least $\binom{k}{2}$ solutions.
\end{lemma}

\begin{proof} Observe that the objective function of $\text{LP}_\text{exact}$ satisfies $0\le \sum_{(i,j) \in E} (1-z_{ij}) +\sum_{(i,j) \not \in E} z_{ij}$ as constraints \eqref{sum},  the McCormick envelopes and $a_{i\ell }, b_{\ell j}\in[0,1]$  imply $z_{ij}\in[0,1]$. Let us construct a feasible solution to  $\text{LP}_\text{exact}$ which attains this bound. Let $a_{i\ell} = b_{\ell j} = \frac{1}{2}$ for all $i\in[n], j\in [m], \ell\in[k]$. The McCormick envelopes then are equivalent to 
	$ MC(\frac{1}{2}, \frac{1}{2}) = \{y \in \R: \frac{1}{2}+\frac{1}{2} -1 \le y, y\le \frac{1}{2}, y\le \frac{1}{2}, 0\le y \} = [0,\frac{1}{2}]$
	hence we may choose the value of $y_{i\ell j} \in MC(\frac{1}{2}, \frac{1}{2})$ depending on the objective coefficient of indices $(i,j)$. 
	For $(i,j)\in E$ the objective function is maximising $z_{ij}$ hence we set $y_{i\ell j}=\frac{1}{2}$ so that the upper bound $ \sum_{\ell=1}^k y_{i\ell j}$  on $z_{ij}$ becomes greater than equal to $1$ and $z_{ij}$ can take value $1$.
	For $(i,j)\not \in E$ the objective function is minimising $z_{ij}$ hence we set $y_{i\ell j}=0$ so that the lower bounds $ y_{i\ell j}\le z_{ij}$ evaluate to $0$ and $z_{ij}$ can take value $0$. Therefore,
	the following setting of the variables shows the lower bound of $0$ on the objective function is attained,
	\begin{align*}
	a_{i\ell} &= \frac{1}{2},\; i\in[n],\ell \in[k];
	&b_{\ell j} &= \frac{1}{2},\;  \ell \in[k],  j\in[m];\\
	y_{i\ell j} &= \frac{1}{2}, (i,j)\in E, \ell \in[k];
	& y_{i\ell j} &= 0,(i,j)\not \in E, \ell \in[k];\\
	z_{ij} &= 1,\; (i,j)\in E;
	& z_{ij} &= 0 ,\; (i,j)\not\in E.
	\end{align*}
	Furthermore, for all $(i,j)\in E$ it is enough to set  $y_{i\ell_1 j}= y_{i\ell_2 j}=\frac{1}{2}$ for only two indices $\ell_1,\ell_2\in [k]$ since this already achieves the upper bound $z_{ij}\le 1 = \sum_{\ell=1}^k y_{i\ell j}$. Hence, there is at least $\binom{k}{2}$ different solutions of $\text{LP}_\text{exact}$ with objective value $0$.
\end{proof}

\subsection{The strength of the LP relaxation of $\text{MIP}_{\text{exact}}$}
\label{bound_on_lp_exact}
For our proof we will need a definition from the theory of binary matrices.
\begin{definition}\cite[Section 2.3]{Monson:1995}
	\label{isol_def}
	Let $X$ be a binary matrix. A set $S\subseteq E=\{ (i,j):x_{ij}=1  \}$ is said to be an \textit{isolated set of ones} if whenever $(i_1,j_1),(i_2,j_2)$ are two distinct members of $S$ then 
	\begin{enumerate}
		\item $i_1\not = i_2$, $j_1 \not =j_2$ and
		\item $x_{i_1,j_2} x_{i_2,j_1}=0$.
	\end{enumerate}
	The size of the largest cardinality isolated set of ones in $X$ is denoted by $i(X)$ and is called the \textit{isolation number} of $X$.
\end{definition}
Observe that requirement (1.) implies that an isolated set of ones can contain the index corresponding to at most one entry in each column and row of $X$. Hence $i(X)\le \min(n,m)$. 
Requirement (2.) implies that if $(i_1,j_1), (i_2,j_2)$ are members of an isolated set of ones then at least one of the entries $x_{i_1,j_2}$, $x_{i_2,j_1}$ is zero, hence members of an isolated set of ones cannot be contained in a common rank-$1$ submatrix of $X$. Therefore if the largest cardinality isolated set of ones has $i(X)$ elements, to cover all $1$s of $X$ we need at least $i(X)$ many rank-$1$ binary matrices in a factorisation of $X$, so $i(X)$ provides a lower bound on the Boolean rank of $X$.

\setcounter{theorem}{1}
\begin{lemma} 
	\label{non-zero-lp}
	Let $X$ be a binary matrix with isolation number $i(X)$. Then for rank-$k$ binary matrix factorisation of $X$ with $k<i(X)$, the LP-relaxation of $\text{MIP}_{\text{exact}}$ has non-zero optimal objective value.
\end{lemma}

\begin{proof}
	Let $k$ be a fixed positive integer and let $X$ be a binary matrix with isolation number $i(X) >k$.
	For a contradiction, assume that $\text{MLP}_{\text{exact}}$ has objective value zero, $\zeta_{\text{MLP}} = 0$.
	\\
	(1.)
	Now, if $\zeta_{\text{MLP}} = 0$ we must have $M_0 \mvec{q} = \mvec{0} = \mvec{\pi}$ in constraint \eqref{zeros} which implies that none of the rank-$1$ binary matrices with $q_\ell>0$ cover zero entries of $X$. In other words, all the rank-$1$ binary matrices active are submatrices of $X$.
	\\
	(2.)
	Now, if $\zeta_{\text{MLP}} = 0$ we must also have $\mvec{\xi} = \mvec{0}$ which implies $M_{\mvec{1}} \mvec{q} \ge \mvec{1} $ in constraint \eqref{ones} for some $\mvec{q} $ which may be fractional but satisfies $\mvec{1}^\top \mvec{q} \le k$. 
	Let $S:=\{(i_1,j_1), \dots, (i_r, j_r), \dots, (i_{|S|}, j_{|S|})\}\subseteq E$ be an isolated set of ones in $X$ of cardinality $i(X)=|S|$.
	Since members of $S$ cannot be contained in a common rank-$1$ submatrix of $X$, all columns $M_\ell$ corresponding to rank-$1$ binary submatrices of $X$ for entries $(i_r,j_r)\in S$ satisfy
	\begin{equation*}
	M_{(i_r,j_r),\ell} =1 \Rightarrow  M_{(i,j),\ell} =0 \;\forall (i,j)\not = (i_r,j_r), (i,j) \in S.
	\end{equation*}
	Therefore, we can partition the active rank-1 binary matrices into $i(X)+1$ groups $G_r$,
	\begin{align}
	G_r &= \{ \ell : q_\ell>0 \text{ and } M_{(i_r,j_r),\ell} =1\} \\
	&  \text{for }\; r=1,\dots,i(X); \nonumber\\
	G_{i(X)+1} &= \{ \ell : q_\ell>0 \text{ and } M_{(i_r,j_r),\ell} =0 \; \forall (i_r,j_r) \in S \}.  
	\end{align}
	While $G_{i(X) + 1}$ may be empty, $G_r$'s for $r=1,\dots,i(X)$ are not empty because we know that
	for all $(i,j)\in E$ we have $\sum_{\ell} M_{(i,j),\ell}\; q_\ell\ge 1$ and $S\subseteq E$.
	Hence for all $(i_r,j_r)\in S$  we have  $\sum_{\ell \in G_r} q_\ell \ge 1$ which implies the contradiction  $ \mvec{1}^\top \mvec{q} \ge \sum_{r=1}^{i(X)} \sum_{\ell \in G_r} q_\ell \ge i(X) > k$ and therefore $\zeta_{\text{MLP}} > 0$.
\end{proof}
Could we replace the condition $k<i(X)$ in Lemma \ref{non-zero-lp}  by a requirement that $k$ has to be smaller than the Boolean rank of $X$? The following example shows that we cannot. 
\setcounter{theorem}{1}
\begin{example}
	Let $X=J_4 - I_4$, where $J_4$ is the $4 \times 4$ matrix of all $1$s and $I_4$ is the $4\times 4$ identity matrix. One can verify that the Boolean rank of $X$ is $4$ and its isolation number is $3$. 
	\begin{equation}
	X =
	\begin{bmatrix}
	0 & 1 & 1 & 1  \\
	1 & 0 & 1 & 1  \\
	1 & 1 & 0 & 1  \\
	1 & 1 & 1 & 0  
	\end{bmatrix}
	\end{equation}
	For $k=3$, the optimal objective value of $\text{MLP}_{\text{exact}}$ is $0$ which is attained by a fractional solution in which the following $6$ rank-$1$ binary matrices are active.
	
	\setlength{\tabcolsep}{3pt} 
	\begin{tabular}{ccc}
		$q_1=\frac{1}{2}$ & $q_2=\frac{1}{2}$ & $q_3=\frac{1}{2}$ \\
		$\begin{bmatrix}
		0 & 0 & 0 & 0  \\
		1 & 0 & 1 & 0  \\
		0 & 0 & 0 & 0  \\
		1 & 0 & 1 & 0  
		\end{bmatrix}$
		&
		$\begin{bmatrix}
		0 & 1 & 1 & 0  \\
		0 & 0 & 0 & 0  \\
		0 & 0 & 0 & 0  \\
		0 & 1 & 1 & 0  
		\end{bmatrix}$
		&
		$\begin{bmatrix}
		0 & 1 & 0 & 1  \\
		0 & 0 & 0 & 0  \\
		0 & 1 & 0 & 1  \\
		0 & 0 & 0 & 0  
		\end{bmatrix}$
	\end{tabular}
	\begin{tabular}{ccc}
		$q_4=\frac{1}{2}$ &$q_5=\frac{1}{2}$ & $q_6=\frac{1}{2}$ \\
		$\begin{bmatrix}
		0 & 0 & 0 & 0  \\
		1 & 0 & 0 & 1  \\
		1 & 0 & 0 & 1  \\
		0 & 0 & 0 & 0  
		\end{bmatrix}$
		&
		$\begin{bmatrix}
		0 & 0 & 1 & 1  \\
		0 & 0 & 1 & 1  \\
		0 & 0 & 0 & 0  \\
		0 & 0 & 0 & 0  
		\end{bmatrix}$
		&
		$\begin{bmatrix}
		0 & 0 & 0 & 0  \\
		0 & 0 & 0 & 0  \\
		1 & 1 & 0 & 0  \\
		1 & 1 & 0 & 0  
		\end{bmatrix}$
	\end{tabular}
\end{example}
%

\subsection{The relation between the LP relaxations of $\text{MIP}_{\text{exact}}$ and MIP($\frac{1}{k}$) }
\label{subsection_coincide}
\begin{lemma}
	For $\rho=\frac{1}{k}$, the optimal objective values of the LP relaxations $\text{MLP}_{\text{exact}}$ and $\text{MLP}(\frac{1}{k})$ coincide. 
\end{lemma}
\begin{proof} 
	It is enough to observe that since $\text{MLP}_{\text{exact}}$ is a minimisation problem, $\mvec{\pi}$ takes the minimal optimal value $\frac{1}{k} M_0 \mvec{q}$ in $\text{MLP}_{\text{exact}}$ due to constraint \eqref{zeros} which equals the second term in the objective \eqref{hamming_obj} of MLP($\frac{1}{k}$). 
\end{proof}
\subsection{Column generation applied to the LP relaxation of the strong formulation of $\text{MIP}_{\text{exact}}$}
\label{strong_lp_sec}

We obtain a modification of $\text{MIP}_{\text{exact}}$ which we call the ``strong formulation" by replacing constraints \eqref{zeros} by exponentially many constraints.
The following is the LP relaxation of the strong formulation of $\text{MIP}_{\text{exact}}$,
\begin{align*}
\label{mlp_strong_exact}
(\text{MLP}_{\text{exact strong}}) &\zeta_{\text{MLP}} = \min  \mvec{1}^\top \mvec{\xi}  +  \mvec{1}^\top\mvec{\pi}  \\
\text{s.t. }
& M_{\mvec{1}} \mvec{q} +\mvec{\xi}\ge \mvec{1}  \\
&(M_{\mvec{0}})_{\ell} \; q_\ell \le \mvec{\pi},\;\;\ell \in [(2^n-1)(2^m-1)] \\
&\mvec{1}^\top \mvec{q} \le k & \\
& \mvec{\xi} \ge \mvec{0}, \mvec{\pi} \in [0,1]^{nm-|E|} , \mvec{q} \in [0,1]^{|\RR|}. 
\end{align*}

\begin{lemma}
	\label{mlp_strong_lemma}
	Applying the CG approach to $\text{MLP}_{\text{exact strong}}$ cannot be used to generate sensible columns.
\end{lemma}

\begin{proof} 
	Let us try applying column generation to solve $\text{MLP}_{\text{exact strong}}$ and add a column of all $1$s as our first column $q_1$. Then at the $1$st iteration, for $q_1 = 1$ the objective value of the Restricted MLP is $\zeta_{\text{RMLP}}^{(1)} = 0 + (nm- |E|)$ for solution vector $[\mvec{\xi}^{(1)}, \mvec{\pi}^{(1)},\mvec{q}^{(1)}] = [ \mvec{0}, \mvec{1}, 1]$. Adding the same column of all $1$s at the next iteration and setting $[q_1,q_2] = [\frac{1}{2}, \frac{1}{2}]$, allows us to keep $\mvec{\xi}^{(2)} = \mvec{0}$ but set $\mvec{\pi}^{(2)} = \frac{1}{2}\mvec{1}$ to get $\zeta_{\text{RMLP}}^{(2)} = 0 + \frac{1}{2}(nm- |E|)$. Therefore continuing adding the same column of all $1$s, after $t$ iterations we have $\zeta_{\text{RMLP}}^{(t)} = 0 + \frac{1}{t}(nm- |E|)$ for solution vector $[\mvec{\xi}^{(t)}, \mvec{\pi}^{(t)},\mvec{q}^{(t)}] = [\mvec{0}, \frac{1}{t}\mvec{1}, \frac{1}{t}\mvec{1}]$. Therefore for $t\to \infty$ we have $\zeta_{\text{RMLP}}^{(t)} \to 0$ and we have not generated any other columns but the all $1$s.
\end{proof}

\subsection{The greedy algorithm for bipartite binary quadratic optimisation}
\label{greedy_punnen}
For a given $n\times m$ coefficient matrix $H$, we aim to find $\mvec{a}\in \B^{n}$ and $ \mvec{b}\in \B^m$ so that $\mvec{a}^\top H \mvec{b}$ is maximised.
Let $\gamma^+_i$ be the sum of the positive entries of $H$ for each row $i\in[n]$, $\gamma^+_i := \sum_{j=1}^m \max(0,h_{ij})$. Reorder the rows of $H$ according to decreasing values of $\gamma^{+}_i$.
Algorithm \ref{alg1} is the greedy heuristic of \cite{Punnen:2012} which provides the optimal solution to the bipartite binary quadratic problem if $\min(n,m)\le2$ and has an approximation ratio of $1/(\min(n,m)-1)$ otherwise. 

\begin{algorithm}
	\caption{Greedy Algorithm}
	\label{alg1}
	\SetAlgoLined
	Phase I. Order $i\in[n]$ so that $\gamma^+_{i} \ge \gamma^+_{i+1}$.\\
	Set $\mvec{a}=\mvec{0}_n$, $\mvec{s}=\mvec{0}_m$. \\
	\For{$i\in[n]$}{
		$f_{0} = \sum_{j=1}^m \max ( 0, s_j)$\\
		$ f_1= \sum_{j=1}^m \max ( 0, s_j+ h_{ij})$\\
		\If{$f_{0} < f_1$}{ 
			Set $a_i = 1$, $\mvec{s} = \mvec{s} + \mvec{h}_i$
		}
	}
	Phase II. Set $\mvec{b}=\mvec{0}_m$.\\
	\For{$j\in[m]$}{
		\If{$(\mvec{a}^\top H)_j>0$}{
			Set $b_j = 1$}
	}
\end{algorithm}
%

\subsection{Rank-$k$ greedy heuristic}
\label{rank-k greedy}
For a given $X\in \B^{n\times m}$ and $k\in \Z_+$, according to Equation \eqref{frob_obj} we may write $\min\|X-Z\|_2^F$ as $|E| - \max \sum_{i\in[n],j\in[m]} h_{ij} z_{ij}$ were $h_{ij}$ are entries of $H:=2X-1$. We propose the  heuristic in Algorithm \ref{ALGO_greedy_rank_k} to compute $k$-BMF by sequentially computing $k$ rank-$1$ binary matrices using the greedy algorithm of \cite{Punnen:2012} given in Algorithm \ref{alg1}. 

We remark that this rank-$k$ greedy algorithm can be used to obtain a heuristic solution to $k$-BMF under standard arithmetic as well: simply modify the last line to $$H[\mvec{a} \mvec{b}^\top==1] = - K$$ for a large enough positive number $K$ (say $K:=\sum_{ij}x_{ij}$) so that each entry of $X$ is covered at most once.

\begin{algorithm}[h]
	\SetAlgoLined
	\caption{\label{ALGO_greedy_rank_k}Greedy algorithm for rank-$k$ binary matrix factorisation}
	Input: $X\in\B^{n\times m}$, $k\in \Z_+$. \\
	Set $H := 2X -1$.\\
	\For{$\ell\in[k]$}{
		$\mvec{a}, \mvec{b} = \text{Greedy} (H)$ \tcp{Compute a rank-1 binary matrix via the greedy algorithm}
		$\mvec{a}, \mvec{b} = \text{Alt} (H, \mvec{a}, \mvec{b})$  \tcp{Improve greedy solution with the alternating heuristic}
		$A_{:,\ell} = \mvec{a}$ \\
		$B_{\ell,:} = \mvec{b}^\top$\\
		$H[\mvec{a} \mvec{b}^\top==1] = 0$ \tcp{Set entries of $H$ to zero that are covered}
	}
	Output: $A\in\B^{n\times k}$, $B\in\B^{k\times m}$
\end{algorithm}

\subsection{Formulations evaluated in Table \ref{gap}}
\label{app_table_2_formulations}
The formulation KGH17 of \cite{Kovacs:2017} with new objective function \eqref{app_obj} for $\rho=1$ that was evaluated to get results in column KGH17 of Table \ref{gap} reads as
\begin{align*}
\min_{a,b, z}
& \sum_{(i,j) \in E} (1-z_{ij} )+ \sum_{(i,j) \not \in E} \sum_{\ell=1}^k y_{i\ell j} \\
\text{s.t. }
&  z_{ij} \le \sum_{\ell=1}^ky_{i\ell j}, ~~~~~~~~~~~~~~\,(i,j)\in E,\\
& y_{i\ell j}\in MC(a_{i\ell},b_{\ell j}),~~~~~i\in[n],\ell\in[k], j\in[m],\\
& z_{ij} \in [0,1], ~~~~~~~~~~~~~~~~~~~~(i,j)\in E,\\
& a_{i\ell},b_{\ell j} \in \B, ~~~~~~~~~~~~i\in[n],\ell\in[k], j\in[m],
\end{align*}
with $MC(a,b) = \{y \in \R: a+b -1 \le y, y\le a, y\le b, 0\le y \}$ denoting the McCormick envelopes as defined in Section \ref{section_formulation}.
In the exponential formulation of \cite{Lu:2008:OBM:1546682.1547186} all possible non-zero binary row vectors $\mvec{\beta}_{t}\in\B^{1\times m}$ ($t\in[2^m-1]$) for factor matrix $B\in\B^{k\times m}$ are explicitly enumerated and treated as fixed input parameters to the formulation. The formulation LVA08  with objective function \eqref{app_obj} for $\rho=1$ that was evaluated to get results in column LVA08 of Table \ref{gap} reads as
\begin{align*}
\min_{\alpha,\delta,z}
&  \sum_{(i,j) \in E} (1-z_{ij} )+ \sum_{(i,j) \not \in E} \sum_{t=1}^{2^m-1} \alpha_{it}\, \beta_{tj}  \\
\text{s.t. }
& z_{ij} \le \sum_{t=1}^{2^m-1} \alpha_{it}\, \beta_{tj} , ~~~~ (i,j)\in E,\\
& \alpha_{it} \le \delta_t, ~~~~~~~~~~~~~~~~~~~~~ i\in[n],t\in[2^m-1],\\
& \sum_{t=1}^{2^m-1} \delta_t \le k, \\
& z_{ij} \in[0,1], ~~~~~~~~~~~~~~~~(i,j)\in E,\\
& \alpha_{i t},\delta_{t} \in \B,  ~~~~~~~~~i\in[n],t\in[2^m-1].
\end{align*}

\subsection{Datasets}
\label{data}
In general if a dataset has a categorical feature $C$ with $N$ discrete options $v_j$, $(j\in [N])$, we convert feature $C$ into $N$ binary features $B_j$ $(j\in N)$ so that if the $i$-th sample takes value $v_j$ for $C$ that is $(C)_i = v_j$, then we have value $(B_j)_i = 1$ and $(B_\ell)_i = 0$ for all $\ell\not = j \in [N]$. This techinque of binarisation of categorical columns has been applied in \cite{Kovacs:2017} and \cite{Barahona:2019}.
The following datasets were used:
\begin{itemize}
	\item 
	The Zoo dataset (\textit{zoo})  \cite{ZOO} describes $101$ animals with $16$ characteristic features. All but one feature is binary. The categorical column which records the number of legs an animal has, is converted into two new binary columns indicating if the number of legs is \textit{less than or equal} or \textit{greater} than four. The size of the resulting fully binary dataset is $101 \times 17$.
	
	\item 
	The Primary Tumor dataset (\textit{tumor}) \cite{TUMOR} contains observations on $17$ tumour features detected in $339$ patients. The features are represented by $13$ binary variables and $4$ categorical variables with discrete options. The $4$ categorical variables are converted into $11$ binary variables representing each discrete option. Two missing values in the binary columns are set to value 0. The final dimension of the dataset is $339 \times 24$.
	
	\item 
	The Hepatitis dataset (\textit{hepat}) \cite{HEP}  consists of 155 samples of medical data of patients with hepatitis. The 19 features of the dataset can be used to predict whether a patient with hepatitis will live or die. 
	6 of the 19 features take numerical values and are converted into 12 binary features corresponding to options: \textit{less than or equal to the median value}, and \textit{greater than the median value}. The column that stores the sex of patients is converted into two binary columns corresponding to labels man and female. The remaining 12 columns take values \textit{yes} and \textit{no} and are converted into 24 binary columns. The raw dataset contains 167 missing values, and according to the above binarisation if a sample has missing entry for an original feature it will have 0's in both columns binarised from that original feature. The final binary dataset has dimension $155 \times 38$.
	
	\item 
	The SPECT Heart dataset (\textit{heart}) \cite{SPECT} describes cardiac Single Proton Emission Computed Tomography images of $267$ patients by $22$ binary feature patterns. $25$ patients' images contain none of the features and are dropped from the dataset, hence the final dimension of the dataset is $242 \times 22$.
	
	\item The Lymphography dataset (\textit{lymp})  \cite{LIM} contains data about lymphography examination of $148$ patients. $8$ features take categorical values and are expanded into $33$ binary features representing each categorical value. One column is numerical and we convert it into two binary columns corresponding to options: \textit{less than or equal to median value}, and \textit{larger than median value}. The final binary dataset has dimension $148\times 44$.
	
	\item 
	The Audiology Standardized dataset (\textit{audio}) \cite{AUDIO} contains clinical audiology records on $226$ patients. The $69$ features include patient-reported symptoms, patient history information, and the results of routine tests which are needed for the evaluation and diagnosis of hearing disorders. $9$ features that are categorical valued are binarised into $34$ new binary variables indicating if a discrete option is selected. The final dimension of the dataset is $226\times 94$.
	
	\item 
	The Amazon Political Books dataset (\textit{books}) \cite{BOOKS} contains binary data about $105$ US politics books sold by Amazon.com. Columns correspond to books and rows represent frequent co-purchasing of books by the same buyers. The dataset has dimension $105\times 105$.
	
	\item 
	The 1984 United States Congressional Voting Records dataset (\textit{votes})\cite{VOTING} includes votes for each of the U.S. House of Representatives Congressmen on the $16$ key votes identified by the CQA. The $16$ categorical variables taking values of ``voted for'', ``voted against'' or ``did not vote'', are converted into $32$ binary variables. One congressman did not vote for any of the bills and its corresponding row of zero is dropped. The final binary dataset has dimension $434\times 32$.

\end{itemize}

\subsection{Data preprocessing}
In practice, the input matrix $X\in\{0,1\}^{n\times m}$ may contain zero rows  or columns. Deleting a zero row (column) leads to an equivalent problem whose solution $A$ and $B$ can easily be translated to a solution of the original dimension. 
In addition, if a row (column) of $X$ is repeated $\alpha_i$ ($\beta_j$) times, it is sufficient to keep only one copy of it, solve the reduced problem and reinsert the relevant copies in the corresponding place. To ensure that the objective function of the reduced problem corresponds to the factorisation error of the original problem, the variable corresponding to the representative row (column) in the reduced problem is multiplied by $\alpha_i$ ($\beta_j$).

\subsection{Comparison Methods }
\label{app_comp}
The following methods were evaluated for the comparison in Table \ref{COMP}.
\begin{itemize}
	\item The code for our methods can be found at \url{https://github.com/kovacsrekaagnes/rank_k_BMF}.
	\item For the alternating iterative local search algorithm  of \cite{Barahona:2019}  (ASSO++) we obtained the code from the author's github page. 
	The code implements two variants of the algorithm and we report the smaller error solution from two variants of it. 
	\item 
	The greedy algorithm ($k$-greedy) detailed in Appendix \ref{rank-k greedy} is evaluated with nine different orderings and the best result is chosen.
	\item For the method of  \cite{Zhang:2007}, we used an implementation in the github pymf package by Chris Schinnerl and we ran it for 10000 iterations.
	\item We evaluated the heuristic method ASSO \cite{Miettinen:2006:PKDD} which  depends on a parameter and we report the best results across nine parameter settings ($\tau \in \{0.1, 0.2,\dots ,0.9 \}$). The code was obtained form the webpage of the author.
	\item  In addition, we computed rank-$k$ non-negative matrix factorisation (NMF) and binarise it by a threshold of $0.5$: after an NMF is obtained, values greater than $0.5$ are set to $1$, otherwise to $0$.  For the computation of NMF we used \texttt{sklearn.decomposition} module in Python. 
	\item For the MEBF method \cite{Wan:2020} we used the code from the author's github page. The raw code downloaded contained a bug and did not produce a solution for some instances while for others it produced factorisations whose error in $\|\cdot\|_F^2$ increased with the factorisation rank $k$. We fixed the code and the results shown in \ref{COMP} correspond to the lowest error for each instance selected across 9 parameter settings $t\in \{0.1,\dots, 0.9\}$.
\end{itemize}